\documentclass[12pt,a4paper,reqno]{amsart}
\usepackage{latexsym,amsmath}
\usepackage[left=3cm,top=2.5cm,right=3cm,bottom=2.5cm]{geometry}
\usepackage[usenames,dvipsnames]{color}
\usepackage{amsthm}
\usepackage{amssymb}
\usepackage{epsfig}
\usepackage{mathrsfs}
\usepackage{verbatim}
\usepackage[normalem]{ulem}
\usepackage{tikz}
\usepackage{hyperref}
\hypersetup{colorlinks=true,citecolor=blue}
\usepackage{lineno}
\usepackage{fullpage}
\usepackage{setspace}
\usepackage{enumerate}

\newcommand{\SC}{\mathcal{SC}}

\def\circcycle{C^\circlearrowright}

\def\R{\mathcal{R}}

\newcommand{\hm}[1]{\leavevmode{\marginpar{\tiny%
$\hbox to 0mm{\hspace*{-0.5mm}$\leftarrow$\hss}%
\vcenter{\vrule depth 0.1mm height 0.1mm width \the\marginparwidth}%
\hbox to 0mm{\hss$\rightarrow$\hspace*{-0.5mm}}$\\\relax\raggedright #1}}}
\newtheorem{theorem}{Theorem}[section]

\newtheorem{lemma}[theorem]{Lemma}
\newtheorem{fact}[theorem]{Fact}

\newtheorem{corollary}[theorem]{Corollary}
\newtheorem{proposition}[theorem]{Proposition}
\newtheorem{problem}[theorem]{Problem}

\theoremstyle{definition}

\newcommand{\oldqed}{}
\def\endofFact{\hfill\scalebox{.6}{$\Box$}}

\def\twins{\mathop{\text{\rm twins}}\nolimits}
\newcommand\normp[2]{\left\Vert #2 \right\Vert_{#1}}

\def\F{\vec{\mathcal{F}}}
\def\R{\vec{\mathcal{R}}}
\def\S{\vec{\mathcal{S}}}

\title{Counting orientations of graphs with no strongly connected tournaments}

\begin{document}
\onehalfspace

\author[F.~Botler]{F\'abio Botler}
\author[C.~Hoppen]{Carlos Hoppen}
\author[G.~O.~Mota]{Guilherme Oliveira Mota}

\address{Programa de Engenharia de Sistemas e Computação, Universidade Federal do Rio de Janeiro, Rio de Janeiro, Brazil}
\email{fbotler@cos.ufrj.br}

\address{Departamento de Matem\'atica Pura e Aplicada, Universidade Federal do Rio Grande do Sul, Porto Alegre, Brasil}
\email{choppen@ufrgs.br}

\address{Instituto de Matem\'atica e Estat\'{\i}stica, Universidade de S\~ao Paulo, S\~ao Paulo, Brazil}
\email{mota@ime.usp.br}

\thanks{\tiny
C.~Hoppen was partially supported by CNPq (308054/2018-0) and FAPERGS (19/2551-0001727-8).
F. Botler was supported by CNPq (423395/2018-1), and by FAPERJ (211.305/2019);
G.~O.~Mota was partially supported by CNPq (304733/2017-2, 428385/2018-4) and FAPESP (2018/04876-1, 2019/13364-7).
This study was financed in part by the Coordena\c{c}\~ao de Aperfei\c{c}oamento de Pessoal de N\'ivel Superior, Brasil (CAPES), Finance Code~001.
FAPERGS is the Rio Grande do Sul Research Foundation.
FAPERJ is the Rio de Janeiro Research Foundation. 
FAPESP is the S\~ao Paulo Research Foundation.  
CNPq is the National Council for Scientific and Technological
Development of Conselho Nacional de Desenvolvimento Cient\'{i}fico e
Tecnol\'{o}gico do Brasil.}

\begin{abstract}
Let $S_k(n)$ be the maximum number
of orientations of an $n$-vertex graph $G$ in which \emph{no copy} of $K_k$ is strongly connected.
For all integers $n$, $k\geq 4$ where $n\geq 5$ or $k\geq 5$, 
we prove that $S_k(n) =  2^{t_{k-1}(n)}$, where $t_{k-1}(n)$ is the number of edges of the $n$-vertex
$(k-1)$-partite Tur\'an graph $T_{k-1}(n)$, and that $T_{k-1}(n)$ is the only $n$-vertex graph with this number of orientations.
Furthermore, $S_4(4) = 40$ and this maximality is achieved only by $K_4$.
\end{abstract}

\maketitle

\section{Introduction}

Let $G$ be a graph and $\vec F$ be an oriented graph.
An orientation $\vec G$ of $G$ is \emph{$\vec F$-free} if $\vec G$ contains no copy of $\vec F$.
Given a family $\F$ of oriented graphs, denote by $\mathcal{D}(G,\F)$ the family of orientations of $G$ that are $\vec F$-free for every $\vec F \in \F$ and write $D(G, \F) = |\mathcal{D}(G,\F)|$.
Finally, let
\begin{equation}\label{def_eq}
D(n,\F)=\max\{D(G,\F) \colon |V(G)|=n\}.
\end{equation}
An $n$-vertex graph that achieves equality in~\eqref{def_eq} is called \emph{$\F$-extremal}.

The \emph{Tur\'an graph}, denoted by $T_{k-1}(n)$, is the $n$-vertex graph without copies of $K_k$ and maximum number of edges possible, which is the balanced complete $(k-1)$-partite graph.
For simplicity, we write $t_{k-1}(n)$ for $|E(T_{k-1}(n))|$.

The problem of determining $D(n,\F)$ has been solved by Alon and Yuster~\cite{AlYu06} for large~$n$ when $\F$ consists of a single \emph{tournament}. 
Recently, Ara\'ujo and the first and last authors~\cite{ArBoMo20+}
extended this result to every $n$ in the case where the
forbidden tournament is the strongly connected triangle, here denoted by $\circcycle_3$.
More precisely, these results are given as follows.
\begin{theorem}
\label{conj:AY}
For $n\geq 1$, we have
	$
	D(n,\{\circcycle_3\}) = \max\{2^{\lfloor n^2/4\rfloor},n!\}.
	$
	Furthermore, among all graphs~$G$ with $n\geq 8$ vertices, $D(G,\{\circcycle_3\}) = 2^{\lfloor n^2/4\rfloor}$ if and only if $G$ is the Tur\'an graph $T_{2}(n)$.
\end{theorem}
More recently, Buci\'c and Sudakov~\cite{BuSu20+} determined $D(n,\{\circcycle_{2k+1}\})$ for strongly connected odd cycles $\circcycle_{2k+1}$ when $n$ is sufficiently large.

Clearly, if $\F$ and $\F'$ are families of oriented graphs such that $\F \subseteq \F'$, then $D(n,\F') \leq D(n,\F)$ for any integer $n$. Thus, the result of Alon and Yuster determining $D(n,\{\vec K_k\})$ for any fixed tournament $\vec K_k$ on $k \geq 3$ vertices and sufficiently large $n$ immediately implies the following: for any nonempty family $\F$ of orientations of $K_k$, we also have $D(n,\F)=2^{t_{k-1}(n)}$ for large $n$. More generally, if $\F$ is a family of tournaments (not necessarily with the same size) and $k$ is the minimum size of a tournament in $\F$, we must also have $D(n,\F)=2^{t_{k-1}(n)}$ for large $n$.
However, to the best of our knowledge, the only nontrivial family $\F$ for which $D(n,\F)$ is known for all values of $n\geq 1$ is $\F=\{\circcycle_3\}$.
We extend Theorem~\ref{conj:AY} to the family $\S_k$ of all strongly connected tournaments on $k\geq 4$ vertices.
More precisely, we prove in Theorem~\ref{thm:Skn} that for all integers $n$, $k\geq 4$ where $n\geq 5$ or $k\geq
5$, we have $D(n,\S_k) =  2^{t_{k-1}(n)}$ and we show that $T_{k-1}(n)$ is the only $\S_k$-extremal graph.
For $n=k=4$, we have $D(n,\S_4) = 40$ and $K_4$ is the only $\S_4$-extremal graph.

We remark that our result implies a similar one for any family $\F$ of complete graphs with at least \(k\) vertices that contains $\S_k$. In fact,  it is enough that $\F$ is a family of graphs  with chromatic number at least $k$ for which $\vec{\mathcal{S}}_k \subset \F$.
More precisely, one can easily conclude that for $n\geq 5$ or $k\geq
5$ we have $D(n,\F) = 2^{t_{k-1}(n)}$ and that $T_{k-1}(n)$ is the only $\F$-extremal graph.
As an application of this, consider the family $\R_k$ of all orientations of $K_k$ that are not transitive, so that $D(n,\R_k)$ is the maximum number of orientations of an $n$-vertex graph for which every copy of $K_k$ is transitively oriented.
Since $\S_k\subset \R_k$, our result implies that for $n\geq 5$ or $k\geq
5$ we have $D(n,\R_k) = 2^{t_{k-1}(n)}$ and $T_{k-1}(n)$ is the unique $\R_k$-extremal graph (for $n=k=4$, it is easy to see that $T_3(4)$ is the unique $\R_4$-extremal graph).
Our results for $\R_k$ and $\S_k$ solve Problems 3 and 4 in~\cite{ArBoMo20+}.

In Section~\ref{sec:multi} we prove general results for orientations avoiding graphs from a family $\F$ of tournaments (see Lemmas~\ref{thm_existence} and~\ref{lemma:containsmultipartite}).
These results imply that if $G$ is an $\F$-extremal graph that is not complete and multipartite,
then one can construct $\F$-extremal complete multipartite graphs $G_1$ and \(G_2\) such that $|E(G_2)| < |E(G)|\leq |E(G_1)|$.
We use these tools to prove our main result in Section~\ref{sec:main}.
We finish with some open problems in Section~\ref{sec:open}.

\section{Multipartite complete extremal graphs}
\label{sec:multi}

In this section we obtain some results derived with the approach in~\cite{BeHoSa17}, 
which in turn was influenced by the proof of Tur\'{a}n's Theorem by Zykov Symmetrization. 
To highlight the similarities, we use the notation of~\cite{BeHoSa17} whenever is possible.
Fix a family $\F$ of tournaments (not necessarily with the same number
of vertices)
on at least three vertices and fix a positive integer $n$.
We prove that there is a complete multipartite $n$-vertex graph that
is $\F$-extremal (see Lemmas~\ref{thm_existence} and~\ref{lemma:containsmultipartite}).

Let $\vec{x}$ be a vector whose coordinates are indexed by a set $T$.
Given $t\in T$, we denote by~$x(t)$ the value of $x$ at coordinate $t$. 
Given $p\in (0, \infty)$, the \emph{$\ell_p$-norm} of $\vec{x}$,
denoted by~$\normp{p}{\vec{x}}$, is given by

$$\normp{p}{\vec x} = \left(\sum_{t\in T} |x(t)|^p\right)^{1/p}.$$

Moreover, for a sequence of vectors $\vec{x_1}, \ldots, \vec{x_s}$, each indexed by $T$, 
their pointwise product $\vec{y}$, i.e., the vector in which $y(t) = \prod_{k=1}^{s}x_k(t)$ for \(t\in T\),
is denoted by $\prod_{k=1}^{s} \vec{x}_k$.

Let $G$ be a graph and let $\F$ be a family of oriented graphs. 
Given a subgraph $H$ of $G$ and an $\F$-free orientation $\vec{H}$ of $H$, we denote by $c_{\F}(G \mid \vec{H})$ the number of ways to
orient the edges in $E(G)\setminus E(H)$ in order to extend $\vec{H}$ to an
$\F$-free orientation of~$G$.
For simplicity, given $v\in V(G)\setminus V(H)$, if $H$ is an induced subgraph of $G$, then we write $c_{\F}(v,\vec{H})$ for $c_{\F}(G[V(H)\cup\{v\}], \vec{H})$.
Similarly, given an edge $\{u,v\}$ of $E(G)\setminus E(H)$, we use $c_{\F}(\{u,v\},\vec{H})$ for the number of ways to orient the edges between $V(H)$ and $\{u,v\}$, and the edge $\{u,v\}$ (again avoiding $\F$),
i.e., the edges in \(E(G[V(H)\cup\{x,y\}])\setminus E(H)\).

We also define  $\vec{v}_{H,\F}$ as the vector indexed by the set $\mathcal{D}(H,\F)$ of all $\F$-free orientations of $H$, whose coordinate corresponding to an orientation $\vec{H}$ is given by $\vec{v}_{H, \F}(\vec{H}) = c_{\F}(v,\vec{H})$.
When \(\F\) is a family of tournaments, we obtain the following simple observation.
\begin{fact}\label{prop:indepset}
Let $\F$ be a family of tournaments.
If $H$ is an induced subgraph of $G$ such that $S = V(G)\setminus V(H)$ is an independent set in $G$, and $\vec{H}$ is an $\F$-free orientation of $H$, then
\[
 	c_{\F}(G\mid \vec{H}) = \prod_{v\in S} c_{\F}(v,\vec{H}).
 \] 
\end{fact} 
We shall also use the following consequence of H\"{o}lder's inequality.
\begin{lemma}\label{cor:holder}
Let $\vec{x}_1, \ldots, \vec{x}_s$ be complex-valued vectors indexed by a set $T$. Then,
$$\normp{1}{\prod_{k=1}^{s} \vec{x}_k} \le \prod_{k=1}^{s}\normp{s}{\vec{x}_k}.$$
Furthermore, equality holds if and only if, for every $i,j \in [s]$, there exists $\alpha_{i,j}$ with the property that $x_i(t)=\alpha_{i,j} x_j(t)$ for all $t \in T$.
\end{lemma}

We say that two non-adjacent vertices of a graph are \emph{twins} if they have the same neighborhood.
For the next lemma we consider the following operation: take an
independent set $S$ of a graph $G$, select a particular vertex $v\in
S$, delete all vertices in $S\setminus\{v\}$ and add $|S|-1$ new twins
of $v$.
We show that there is a vertex $v \in S$ for which this operation produces a graph that has at least as many $\F$-free orientations as $G$.

\begin{lemma}\label{lemma:replaceIndSet}
Let $\F$ be a family of tournaments, and let $G$ be a graph on $n$ vertices.
If $S \subset V(G)$ is a non-empty independent set, then the following holds.
\begin{enumerate}[(a)]
\item \label{lemma:replaceIndSeta}
if \(v\) is a vertex in \(S\) for which
$\normp{|S|}{\vec{v}_{G-S,\F}}$ is maximum among all vertices of $S$,
then the graph \(\widetilde{G}\) 
obtained from \(G\) by replacing the vertices in \(S\setminus\{v\}\) with \(|S|-1\) twins of \(v\) is such that $D(\widetilde{G},\F)\geq D(G,\F)$; and
\item \label{lemma:replaceIndSetb}
if $G$ is $\F$-extremal, then $\vec{u}_{G-S,\F} = \vec{w}_{G-S,\F}$ for any vertices $u, w \in S$.
\end{enumerate}
\end{lemma}

\begin{proof}
Let \(\F\), \(G\), and \(S\) be as in the statement, and put \(H=G-S\)
and $s = |S|$.
By Fact~\ref{prop:indepset}, the total number of $\F$-free orientations of $G$ is given by
\begin{equation} \label{eq:primeira}
D(G,\F) = \sum_{\vec{H}\in D(H,\F)} c_{\F}(G \mid \vec{H}) = \sum_{\vec{H}\in D(H,\F)} \prod_{u\in S} c_{\F}(u, \vec{H})= \normp{1}{\prod_{u\in S} \vec{u}_{H,\F}},
\end{equation}
where, in the last equality, we used the fact that every coordinate of $\vec{u}_{H,\F}$ is non-negative.

Let $v$ be a vertex in $S$ for which $\normp{s}{\vec{v}_{H,\F}}$ is maximum. 
By Lemma~\ref{cor:holder}, we have

\begin{equation} \label{eq:usingholder}
\normp{1}{\prod_{u\in S} \vec{u}_{H,\F}} \le \prod_{u \in S}\normp{s}{\vec{u}_{H,\F}} 
\le \normp{s}{\vec{v}_{H,\F}}^s.
\end{equation}
On the other hand, let  \(\widetilde{G}\) be the graph
obtained from \(G\) by replacing the vertices in \(S\setminus\{v\}\) by \(s-1\) twins of \(v\).
Then, we have:
\begin{equation} \label{eq:terceira}
D(\widetilde{G},\F) = \sum_{\vec{H}\in D(H,\F)} c_{\F}(v, \vec{H})^s =
\normp{s}{\vec{v}_{H,\F}}^s.
\end{equation}
Therefore, combining~\eqref{eq:primeira},~\eqref{eq:usingholder}
and~\eqref{eq:terceira}, we have $D(\widetilde{G},\F) \ge D(G,\F)$. This proves~\eqref{lemma:replaceIndSeta}.

Now, assume $G$ is $\F$-extremal. 
Since $D(\widetilde{G},\F) \ge D(G,\F)$, we have $D(\widetilde{G},\F) = D(G,\F)$. 
This implies that both inequalities in \eqref{eq:usingholder} hold with equality,
and hence for every vertex $u \in S$, we must have $\normp{s}{\vec{u}_{H,\F}} = \normp{s}{\vec{v}_{H,\F}}$. 
By the equality conditions of Lemma \ref{cor:holder}, together with the fact that all coordinates in our vectors are nonnegative, we have $\vec{u}_{H,\F} = \vec{v}_{H,\F}$.
This proves~\eqref{lemma:replaceIndSetb}.
\end{proof}

\begin{corollary}\label{corollary:cloning}
Let $\F$ be a family of tournaments, and let $G$ be an $\F$-extremal graph.
If $u, v\in V(G)$ are non-adjacent, 
then the graph obtained from \(G\) by replacing $v$ with a twin of $u$ is also $\F$-extremal.
\end{corollary}

\begin{proof}
Let $G_{uv}=G-\{u,v\}$. Since $G$ is $\F$-extremal, by Lemma~\ref{lemma:replaceIndSet}\eqref{lemma:replaceIndSetb} with $S = \{u, v\}$,
 we have $\vec{u}_{G_{uv}, \F} = \vec{v}_{G_{uv},\F}$.
 Therefore, by Lemma~\ref{lemma:replaceIndSet}\eqref{lemma:replaceIndSeta},
 replacing $v$ with a twin of $u$ does not change the number of $\F$-free orientations of the graph.
\end{proof}

Note that a graph $G$ is a complete multipartite graph if and only if, for any $u,v,w \in V(G)$ such that $\{u,v\}, \{u,w\} \notin E(G)$, we have $\{v,w\} \notin E(G)$.
By repeatedly applying Corollary~\ref{corollary:cloning}, we show that \emph{there exists} a complete multipartite graph on $n$ vertices that is $\F$-extremal.
\begin{lemma}\label{thm_existence}
Let $\F$ be a family of tournaments and let $G$ be an $n$-vertex $\F$-extremal graph.
Then there exists an $n$-vertex complete multipartite graph $G^\ast$ that is $\F$-extremal and satisfies $|E(G^*)| \geq |E(G)|$.
\end{lemma}

\begin{proof}
  Let $G$ be an $n$-vertex $\F$-extremal graph.
  A vertex \(u\) is \emph{eccentric} if there is a non-neighbor of
  $u$ that is not its twin.
  Put $G_0 = G$ and let \(G_0,\ldots,G_t\) be a maximal sequence of graphs such that,
    for \(i=0,\ldots,t-1\), the graph \(G_{i+1}\) is obtained from \(G_i\) 
    by picking an eccentric vertex \(u\) of maximum degree in $G_i$
	and replacing \textit{every} non-neighbor of \(u\) with a twin of \(u\).
        Thus, every vertex of \(G_{i+1}\) is either a neighbor or a twin of~\(u\).
        Note that in $G_{i+1}$ the vertex $u$ is not eccentric.
	Moreover, if a vertex \(v\) is eccentric in
        \(G_{i+1}\), then \(v\) is also eccentric in \(G_i\),
	and hence \(G_{i+1}\) contains fewer eccentric vertices than
        \(G_i\), from which we conclude that the sequence
        \(G_0,\ldots,G_t\) is finite.

        By construction, since we always pick a vertex $u$ with
        maximum degree,  \(|E(G_t)|\geq |E(G)|\).
        Then, by Corollary~\ref{corollary:cloning}, \(G_t\) is \(\F\)-extremal.
	Moreover, every pair of non-adjacent vertices \(u\) and \(v\) in \(G_t\) are twins.
	Therefore, \(G_t\) is a multipartite complete graph. 	
\end{proof}

The following result implies that any $\F$-extremal graph may also be turned into an $\F$-extremal multipartite graph by removing edges. 

\begin{lemma} \label{lemma:containsmultipartite}
Let $\F$ be a family of tournaments, let $G$ be an $\F$-extremal graph,
and let \(u,v,w\) be distinct vertices of \(G\) such that $\{u,v\}, \{u,w\} \notin E(G)$ and $\{v,w\} \in E(G)$.
Then, the graph obtained from \(G\) by deleting the edge $\{v,w\}$ is $\F$-extremal.
Furthermore, for \(H=G-\{u,v,w\}\) we have $\vec{u}_{H,\F} = \vec{w}_{H,\F} = \vec{v}_{H,\F}$.
\end{lemma}

\begin{proof}
Let \(\F\), $G$, \(u\), \(v\), and \(w\) be as in the statement. 
Let $H = G - \{u, v, w\}$, and write $H^x = G[V(H)\cup x]$ for $x \in
\{u,v,w\}$ and let $G^u_{\twins}$ ($G^w_{\twins}$) be the graph obtained
from $H^u$ by adding twins $u_1$, $u_2$ of $u$ (twins $w_1$, $w_2$ of $w$). 
By Corollary \ref{corollary:cloning} applied twice, the graph $G^u_{\twins}$ is also $\F$-extremal.
Therefore, $ D(G,\F) = D(G^u_{\twins},\F)$.

Applying Proposition \ref{prop:indepset} to $G^u_{\twins}$ with $S = \{u,u_1,u_2\}$, we have

\begin{equation*}
D(G^u_{\twins},\F)= \sum_{\vec{H}\in D(H,\F)} c_{\F}(G^u_{\twins} \mid \vec{H}) = \sum_{\vec{H}\in D(H,\F)} c_{\F}(u,\vec{H})^3 = \normp{3}{\vec{u}_{H,\F}}^3.
\end{equation*}

By an analogous computation, we have $D(G^w_{\twins},\F)=\normp{3}{\vec{w}_{H,\F}}^3$. 
However, since $w$ is neighbor of $u$ and $v$,
Corollary~\ref{corollary:cloning} cannot be applied to the graph
$G^w_{\twins}$, so it is not possible to conclude whether $G^w_{\twins}$ is $\F$-extremal.
But since $G^u_{\twins}$ is $\F$-extremal, we have

\begin{equation}
\label{eq:normUvsW}
\normp{3}{\vec{w}_{H,\F}}^3 = D(G^w_{\twins},\F) \le D(G^u_{\twins},\F)=\normp{3}{\vec{u}_{H,\F}}^3.
\end{equation}

Since there are no edges between $u$ and $\{v, w\}$, we can compute $D(G,\F)$ as follows:

\begin{align*}
D(G,\F) = \sum_{\vec{H}\in D(H,\F)}\left( c_{\F}(u,\vec{H})  c_{\F}(G - u \mid \vec{H})\right)
        = \sum_{\vec{H}\in D(H,\F)}\left( c_{\F}(u,\vec{H})  \sum_{\vec{H^w} \mid \vec{H}} c_{\F}(v, \vec{H^w})\right),
\end{align*}
where the inner sum is taken over the $\F$-free orientations of $H^w$
that extend a given $\F$-free orientation of $H$, that is, over the
orientations of the edges between $w$ and $H$, for which the resulting
orientation is $\F$-free. By Lemma~\ref{lemma:replaceIndSet}(b), since
$G$ is $\F$-extremal and $\{u,v\} \notin E(G)$, we have
$\vec{v}_{H^{w}, \F} = \vec{u}_{H^{w}, \F}$, i.e.,
$c_{\F}(v,\vec{H^w}) = c_{\F}(u,\vec{H^w})$ for every
$\vec{H^w}$. Finally, note that since $\F$ is a family of tournaments
and $u$ and $w$ are not neighbors, $c_{\F}(u,\vec{H^w})$ does not depend on the orientation of the edges between $w$ and $H$, so $c_{\F}(u,\vec{H^w}) = c_{\F}(u,\vec{H})$.
Therefore,

\begin{align}	 
D(G,\F) &= \sum_{\vec{H}\in D(H,\F)}\left( c_{\F}(u,\vec{H}) 
          \sum_{\vec{H^w} \mid \vec{H}} c_{\F}(u, \vec{H})
          \right)
         = \sum_{\vec{H}\in D(H,\F)}\left( c_{\F}(u,\vec{H}) 
           c_{\F}(u,\vec{H})  \sum_{\vec{H^w} \mid \vec{H}} 1 \right) \nonumber \\
	 &= \sum_{\vec{H}\in D(H,\F)}
           c_{\F}(u,\vec{H})^2c_{\F}(w,\vec{H}) \nonumber \\
	 &\le \normp{3}{\vec{u}_{H,\F}} \normp{3}{\vec{u}_{H,\F}} \normp{3}{\vec{w}_{H,\F}} \label{eq:usingHolderB}\\
	 &\le \normp{3}{\vec{u}_{H,\F}}^3. \label{eq:usingHolderC}
\end{align}

Note that~\eqref{eq:usingHolderB} follows by~Lemma~\ref{cor:holder}, 
and \eqref{eq:usingHolderC} follows from~\eqref{eq:normUvsW}. Finally, 
since $D(G,\F) = D(G^u_{\twins},\F) = \normp{3}{\vec{u}_{H,\F}}^3$, we must
have equality in both \eqref{eq:usingHolderB} and
\eqref{eq:usingHolderC}, which in turn leads to
$\normp{3}{\vec{u}_{H,\F}} = \normp{3}{\vec{w}_{H,\F}}$. The equality
condition in Lemma~\ref{cor:holder} implies that $\vec{u}_{H,\F} =
\vec{w}_{H,\F}$. Analogously, $\vec{u}_{H,\F} = \vec{v}_{H,\F}$.

Finally, let $G^-$ be the graph obtained from $G$ by deleting the edge
$\{v,w\}$.
It follows that

$$D(G^u_{\twins},\F)= \sum_{\vec{H}} c_{\F}(u,\vec{H})^3 =\sum_{\vec{H}} c_{\F}(u,\vec{H}) c_{\F}(v,\vec{H}) c_{\F}(w,\vec{H}) = D(G^-,\F).$$

Therefore, $G^-$ is $\F$-extremal. This concludes the proof.
\end{proof}

An interesting consequence of the main results of this section (Lemmas~\ref{thm_existence} and~\ref{lemma:containsmultipartite}) is that
if $G$ is an $\F$-extremal graph that is not complete multipartite,
then one can construct $\F$-extremal complete multipartite graphs $G_1$ and \(G_2\) such that $|E(G_2)| < |E(G)|\leq |E(G_1)|$.

We finish this section proving that for some families $\F$ of
forbidden tournaments \emph{every} $\F$-extremal graph is a
multipartite complete graph.
Given an oriented graph $\vec G$, we write $(v,w)$ to denote an edge oriented from $v$ to $w$ in $\vec G$.

\begin{lemma}\label{lemma:multi}
Let $\F$ be a family of tournaments with no source and let $n$ be a positive integer.
Then, every $n$-vertex $\F$-extremal graph is complete multipartite.
\end{lemma}

\begin{proof}
Let $n\geq 4$ be an integer and let $\F$ as in the statement.
Let $G$ be an $n$-vertex $\F$-extremal graph and assume for a contradiction that $G$ is not complete multipartite.
Fix vertices $u, v, w$ such that $\{u,v\}, \{u,w\} \notin E(G)$ and $\{v,w\} \in E(G)$.

Let $H = G - \{u, v, w\}$, and $H^x = G[V(H)\cup x]$ for $x \in \{u,v,w\}$. From Lemma~\ref{lemma:containsmultipartite}, we have $\vec{u}_{H,\F} = \vec{w}_{H,\F}  = \vec{v}_{H,\F}$, so for every orientation $\vec{H}$ of $H$ we have $c_{\F}(u,\vec{H}) = c_{\F}(w,\vec{H}) = c_{\F}(v,\vec{H})$.
Note that since $\F$ is a family of tournaments and $u$ and $w$ are not adjacent, for every extension of $\vec{H}$ to an orientation $\vec{H^w}$, we must have $c_{\F}(u,\vec{H^w}) = c_{\F}(u,\vec{H})$. Finally, since $u$ and $v$ are not adjacent, by Lemma~\ref{lemma:replaceIndSet} \eqref{lemma:replaceIndSetb}, we have $\vec{u}_{H^w,\F} = \vec{v}_{H^w,\F}$, that is, $c_{\F}(u, \vec{H^w}) = c_{\F}(v, \vec{H^w})$ for every orientation $\vec{H^w}$ of $H^w$. It follows that, for every $\F$-free extension $\vec{H^w}$ of $\vec{H}$, we must have 
   	\begin{equation} \label{eq:extension}
   	c_{\F}(v,\vec{H^w}) = c_{\F}(v,\vec{H}).
   	\end{equation}	

We will get a contradiction from this fact (which implies that this graph $G$ cannot exist). We only need to find an orientation of $\vec{H}$ and an extension of it to $H^w$, which is $\F$-free and such that~\eqref{eq:extension} does not hold. By the definition of $\F$, we may start with a transitive orientation $\vec{H}$ of $H$, which can be extended to an $\F$-free orientation $\vec{H^w}$ of $H^w$ by orienting all edges $\{x,w\}$ between $w$ and $x \in V(H)$ as $(w,x)$.  Let $\mathcal{H}(v)$ and $ \mathcal{H}^w(v)$ be the classes of all orientations that extend $\vec{H}$ to $H^v$ and $\vec{H^w}$ to $G-u$, respectively. We show that there is an injective mapping $\phi \colon \mathcal{H}(v) \rightarrow \mathcal{H}^w(v)$ that is not surjective. 

Given an orientation $\vec{H^v} \in \mathcal{H}(v)$, let $\phi(\vec{H^v})$ be the orientation of $G-u$ that extends $\vec{H^w}$ by orienting any edge $e=\{v,x\}$ between $v$ and $x \in V(H)$ the same orientation as $e$ in $\vec{H^v}$ and by assigning the orientation $(w,v)$ to $\{v,w\}$. The function $\phi$ is clearly injective. We claim that $\phi(\mathcal{H}(v)) \in  \mathcal{H}^w(v)$. To see why this is true, suppose that $\phi(\mathcal{H}(v))$ contains a tournament $\vec T$ with no source. Clearly, this copy involves both $v$ and $w$, otherwise it would also occur in $\vec{H^v}$, a contradiction. However, $w$ is a source in $\phi(\mathcal{H}(v))$, so it cannot lie in $\vec T$.

On the other hand, any transitive orientation of $G-u$ where $\{v,w\}$ is oriented $(v,w)$ and all edges $\{x,v\}$ and $\{y,w\}$ with $x,y \in V(H)$ are oriented $(v,x)$ and $(w,y)$, respectively,  must lie in $\mathcal{H}^w(v)$. However, it does not lie in $\phi(\mathcal{H}(v))$, as $\phi$ always orients $\{v,w\}$ as $(w,v)$. So $\phi$ is not surjective and we have reached the desired contradiction.
\end{proof}

\section{Avoiding strongly connected tournaments}
\label{sec:main}

We start by giving some results concerning Hamilton cycles in oriented graphs.
We refer to a directed Hamilton cycle (resp. directed path) in an oriented graph simply as Hamilton cycle (path).
The following basic fact about Hamilton paths and cycles is useful in our proof.  

\begin{fact}\label{ham_path}
Every tournament contains a Hamilton path and every strongly connected tournament contains a Hamilton cycle.
\end{fact}

By using Fact~\ref{ham_path} we can guarantee the existence of strongly connected tournaments of any length in strongly connected tournaments.

\begin{lemma}\label{lemma:all-small-sc-tournaments}
	Every strongly connected tournament \(\vec K\) contains a strongly connected tournament of order \(\ell\) for every \(3\leq \ell\leq |V(\vec K)|\).
\end{lemma}

\begin{proof}
  Let $\vec K$ be a strongly connected tournament.
  By Fact~\ref{ham_path}, $\vec K$ contains a Hamilton cycle.
	The proof is by induction on \(n=|V(\vec K)|\).
	Clearly, if \(n=3\), then \(\vec K\) is a strongly connected orientation of \(K_3\), and the statement follows.
	Thus, we assume \(n\geq 4\).
	Let \(C = (x_1 x_2 \cdots x_n x_1)\) be a Hamilton cycle in \(\vec K\).
	If \((x_i, x_{i+2})\in E(\vec K)\) for some $i \in \{1,\ldots,n\}$, where indices are taken modulo $n$, then \(C' = (x_1 \cdots x_i x_{i+2} \cdots x_n x_1)\) is
	a Hamilton cycle of \(\vec K'\), the tournament obtained from $\vec K$ by removing $x_{i+1}$.
        
	By the induction hypothesis, \(K'\) contains a strongly connected tournament of order \(\ell\) for every \(3\leq \ell\leq |V(K')|=n-1\) and the statement follows for \(K\).
	Thus, we may assume that \((x_{i+2},x_i)\in E(\vec K)\) for every \(i\in[n]\) (indices taken modulo \(n\)).
	In this case, for each \(3\leq \ell\leq n\), 
	the set \(V_\ell = \{x_1,\ldots,x_\ell\}\) induces a strongly connected tournament of order \(\ell\) 
	because \(\{x_i,x_{i+1},x_{i+2}\}\) induces a strongly connected orientation of \(K_3\) for every $i \in \{1,\ldots,\ell-2\}$.
      \end{proof}

      Recall that \(\S_k\) is the family of all strongly connected tournaments with
$k$ vertices.
Let $K$ be a complete subgraph of a graph $G$ and let $\vec K$ be an $\S_k$-free orientation of $K$.
Given vertices $u$ and $v$ of $G$, recall that we write $c_{\S_k}(u,\vec K)$ for the number of orientations of the set of edges $\{u,w\}$ with $w\in K$ that extend $\vec{K}$ keeping the $\S_k$-freeness of the orientation.
In the next lemma we present a bound on $c_{\S_k}(u,\vec K)$ for particular sizes of $K$ and particular vertices $u$.

\begin{lemma}\label{lemma:cliqueExtension}
  Let $K$ be a complete subgraph of a graph $G$ with $x\in \{k-1,k,k+1\}$
  vertices and let $u$ be a vertex in $V(G)\setminus V(K)$ that is
  adjacent to all vertices of $K$.
  Then, for any orientation $\vec K$ of $K$, we have
  $
  c_{\S_k}(u,\vec K) \leq (x-k+4)\cdot 2^{k-3}.
  $
\end{lemma}

\begin{proof}
We assume $x=k+1$ as the proof for the other values of $x$ is analogous.
Let \(K\), \(G\) and \(u\) be as in the statement.
Let $\vec K$ be an orientation of $K$ and note that by Fact~\ref{ham_path} there is a Hamilton path $(v_1,\ldots,v_{k+1})$ in $\vec K$.
By Lemma~\ref{lemma:all-small-sc-tournaments}, any orientation of the edges between $u$ and $\vec K$ that
extends $\vec K$ to an $\S_k$-free orientation cannot form a directed cycle of length \(k\).
If \(\{u,v_1\}\) is oriented towards \(v_1\),
then the edges between any \(w\) in \(\{v_{k-1}, v_k, v_{k+1}\}\) and $u$ must be oriented towards $w$. 
Since the edges \(\{u,v_j\}\) with \(j\in\{2,\ldots,k-2\}\) can be oriented in two ways,
there are at most \(2^{k-3}\) such possible orientations.
If \(\{u,v_1\}\) is oriented towards \(u\) and $\{u,v_2\}$ is oriented towards $v_2$,
then the edges between any \(w\) in \(\{v_k, v_{k+1}\}\) and $u$ are oriented towards $w$. 
Again, since the edges \(\{u,v_j\}\) with \(j\in\{3,\ldots,k-1\}\) can be oriented in two ways,
there are at most \(2^{k-3}\) such possible orientations.
Analogously, if $\{u,v_1\}$ and $\{u,v_2\}$
are oriented towards $u$ and $\{u,v_3\}$ is oriented towards $v_3$,
then there are at most $2^{k-3}$ such possible orientations.
Finally, there are at most $2^{k-2}$ such  possible orientations
for which $\{u,v_1\}$,  $\{u,v_2\}$ and $\{u,v_3\}$ are oriented towards $u$.
Therefore, 
$ c_{\S_k}(u,\vec K_x) \leq 2^{k-3}+ 2^{k-3} + 2^{k-3} + 2^{k-2} = 5\cdot 2^{k-3}$, as desired.
\end{proof}

As before, let $K$ be a complete subgraph of a graph $G$ and let $\vec K$ be an $\S_k$-free orientation of $K$.
Given an edge $\{u,v\}$ of $G$, similar to Lemma~\ref{lemma:cliqueExtension}, we now provide a bound on $c_{\S_k}(\{u,v\},\vec K)$, where we recall that $c_{\S_k}(\{u,v\},\vec K)$ stands for the number of ways to extend $\S_k$-free orientations by orienting the edges in  $E(G[V(K)\cup\{u,v\}])\setminus E(K)$.

\begin{lemma}\label{lemma:cliqueExtension-edge}
  Let $K$ be a complete subgraph of a graph $G$ with \(k-1\)
  vertices and let $u$, $v$ be adjacent vertices in $V(G)\setminus V(K)$ that are
  adjacent to all vertices of $K$.
  Then,
  $
  c_{\S_k}(\{u,v\},\vec K) <2\cdot 3\cdot 2^{2k-5}.
  $
\end{lemma}

\begin{proof}
Let \(K\), \(G\) and \(u\) be as in the statement.
Let $\vec K$ be an orientation of $K$ and note that by Fact~\ref{ham_path} there is a Hamilton path $(v_1,\ldots,v_{k-1})$ in $\vec K$.
There are two possible orientations for the edge $\{u,v\}$.
Let us estimate in how many ways one can orient the edges between $\vec K$ and $\{u,v\}$ without creating a strongly connected $K_k$.
Suppose, without loss of generality, that \(\{u,v\}\) is oriented towards \(u\).

Note that by Lemma~\ref{lemma:all-small-sc-tournaments}, any orientation of the edges between $u$ and $\vec K$ that
extends $\vec K$ to an $\S_k$-free orientation cannot form a directed cycle of length \(k\).
Analogously to the proof of Lemma~\ref{lemma:cliqueExtension}, 
if \(\{u,v_1\}\) is oriented towards \(v_1\),
then the edge between any \(v_{k-1}\) and $u$ must be oriented towards $v_{k-1}$,
and the edges between any \(w\) in \(\{v_{k-2},v_{k-1}\}\) and $v$ must be oriented towards $w$.
Since the remaining edges from \(u\) or \(v\) to \(K\) can be oriented in two ways,
there are at most \(2^{2k-6}\) possible orientations.
Similarly, 
there are at most \(2^{2k-5}\) possible orientations
in which \(\{u,v_1\}\) is oriented towards \(u\) and $\{u,v_2\}$ is oriented towards $v_2$,
and there are at most \(2^{2k-6}\) possible orientations
in which \(\{u,v_1\}\) and \(\{u,v_2\}\) are oriented towards \(u\),
and \(\{v,v_1\}\) is oriented towards \(v_1\).
Finally, there are at most $2^{2k-5}$ possible orientations
in which \(\{u,v_1\}\) and \(\{u,v_2\}\) are oriented towards \(u\),
and \(\{v,v_1\}\) is oriented towards \(v\).
Therefore, 
$ c_{\S_k}(uv,\vec K_x) \leq 2 \cdot (2^{2k-6} + 2^{2k-5} + 2^{2k-6} + 2^{2k-5}) 
= 2\cdot 3\cdot 2^{2k-5}$.
\end{proof}

The next result states that for \(k\geq 5\), at least half of the orientations of \(K_k\) are strongly connected.
Let \(\SC(G)\) denote the number of strongly connected orientations of \(G\).

\begin{lemma}\label{lemma:most-of-the-orientations-are-strongly-connected}
	Let \(k\geq 5\) be a positive integer.
	Then \(\SC(K_k)> 2^{{k\choose 2} - 1}\).
\end{lemma}

\begin{proof}
The proof is by induction on \(k\).
Let \(\{u_1,\ldots, u_k\}\) be the vertex set of \(G=K_k\) and consider the complete graph \(G'=G-u_k\) on $k-1$ vertices.
First, suppose \(k=5\). We prove that $\SC(K_5)< 512$.
It is easy to see that out of the $40$ not strongly connected orientations of $G'$, there are $24$ transitive orientations and $16$ non-transitive orientations.
Each of these non-transitive orientation contains a directed $C_3$,
and there are $7$ ways to extend this directed $C_3$ to obtain a strongly connected orientation of $G$.
Noting that there are $24$ strongly connected orientations of $G'$, we have

\begin{equation*}
\SC(K_5)   = 24\cdot 4 + 24\cdot (2^4-2) + 16\cdot 7 = 544 > 2^{{5\choose 2} - 1}.
\end{equation*}
    
Now, let $k\geq 6$ and suppose that \(\SC(K_{k-1})> 2^{{k-1\choose 2} - 1}\).
Let \(\vec{G'}\) be an orientation of \(G'\).
If \(\vec{G'}\) is strongly connected, 
then we obtain a strongly connected orientation of \(G\) if and only if
the remaining edges are oriented so that $u_k$ has in degree and out degree at least 1.
Therefore, \(\vec{G'}\) can be extended in precisely \(2^{k-1}-2\) ways
to strongly connected orientation of \(G\).
Thus, we may assume that \(\vec{G'}\) is not strongly connected.
By Fact~\ref{ham_path}, \(G'\) contains a Hamilton path \(P\).
We may assume, without loss of generality, that \(P=(u_1,\ldots, u_{k-1})\).
By orienting \(\{u_1,u_k\}\) towards \(u_1\) and \(\{u_{k-1},u_k\}\) towards \(u_k\),
and orienting the edges \(\{u_i,u_k\}\) in any direction, for \(i=1,\ldots,k-2\),
we obtain a strongly connected orientation of \(G\).
Thus, there are at least \(2^{k-3}\) strongly connected orientations of \(G\) from \(\vec{G'}\).
Therefore, we have

\begin{align*}
\SC(K_k)  & \geq (2^{k-1}-2)\cdot \SC(K_{k-1}) + 2^{k-3}\cdot (2^{{k-1}\choose 2} - \SC(K_{k-1})) \\
& = (2^{k-1}-2 - 2^{k-3})\cdot \SC(K_{k-1}) + 2^{k-3} 2^{{k-1}\choose 2} \\
& > 2^{k-3} 2^{{k-1}\choose 2} + 2^{k-2} 2^{{k-1 \choose 2}-1}  + (2^{k-2} - 2 - 2^{k-3}) \SC(K_{k-1})\\
  & = 2^{{k \choose 2} - 1} + (2^{k-2} - 2 - 2^{k-3})2^{{k-1 \choose 2}-1}\\
  & > 2^{{k \choose 2} - 1},
\end{align*}
which concludes the proof.
\end{proof}

Given a graph $G$, we denote by $\S_k(G)$ the family of $\S_k$-free orientations of $G$ and we write $S_k(G) = |\S_k(G)|$.
Combining some of the previous results, we prove that cliques of size \(k\) and \(k+1\) are not \(\S_k\)-extremal.

\begin{corollary}\label{corollary:small-cliques}
	For \(k\geq 5\) we have \(S_k(K_k) < 2^{t_{k-1}(k)}\) and 
	for \(k\geq 4\) we have \(S_k(K_{k+1}) < 2^{t_{k-1}(k+1)}\).
\end{corollary}

\begin{proof}
We start by showing that $S_4(K_5) < 2^{t_3(5)} = 2^8$.
Every non-transitive orientation of $K_5$ contains a $\circcycle_3$.
We claim that there are precisely ${5\choose 2}\cdot 2\cdot 6$ such non-transitive orientations without a strongly connected $K_4$.
In fact, there are ${5\choose 2}$ triangles in a $K_5$ and $2$ strongly connected orientations of each triangle.
It is not hard to see that there are at most $6$ ways to orient the edges outside the triangle to obtain an orientation of $K_5$ with no strongly connected $K_4$.
Since there are $120$ transitive orientations of $K_5$, we obtain $S_4(K_5)\leq 120 + {5\choose 2}\cdot 2\cdot 6 = 240 < 2^8$.

Thus, assume $k\geq 5$.
Note that a strongly connected orientation of \(K_{k-1}\) 
can be extended in only two (resp. six) ways to an orientation of \(K_k\) (resp. \(K_{k+1}\)) 
that does not contain a strongly connected $K_k$.
By Lemma~\ref{lemma:cliqueExtension}, every orientation of \(K_{k-1}\)
can be extended in at most \(3\cdot 2^{k-3}\) ways 
to an \(\S_k\)-free orientation of \(K_k\), which gives

\begin{equation}\label{eq:skkk}
S_k(K_k)  \leq 2\cdot\SC(K_{k-1}) + 3\cdot 2^{k-3} \big(2^{{k-1\choose 2}} - \SC(K_{k-1})\big).
\end{equation}

To obtain an estimate for $S_k(K_{k+1})$ we use Lemma~\ref{lemma:cliqueExtension-edge}, which implies that every orientation of \(K_{k-1}\)
can be extended in at most \(2\cdot 3\cdot 2^{2k-5}\) ways 
to an \(\S_k\)-free orientation of \(K_{k+1}\).
Then,

\begin{equation}\label{eq:skkk1}
  S_k(K_{k+1})	 \leq 6\cdot\SC(K_{k-1}) + 2\cdot 3\cdot 2^{2k-5} \big(2^{{k-1\choose 2}} - \SC(K_{k-1})\big).
\end{equation}

If $k=5$, then, by Lemma~\ref{lemma:most-of-the-orientations-are-strongly-connected}, 
we have $\S_5(K_5) = 2^{{5 \choose 2}} - \SC(K_5) < 2^{9} = 2^{t_4(5)}$.
From~\eqref{eq:skkk1}, since $\SC(K_4) = 24$, we obtain $\S_5(K_6) \leq 6\cdot 24 + 2\cdot 3\cdot 2^5(2^6 - 24) = 7824 < 2^{13} = 2^{t_{4}(6)}$.

Now we assume that $k\geq 6$.
From~\eqref{eq:skkk}, using Lemma~\ref{lemma:most-of-the-orientations-are-strongly-connected} and \(t_{k-1}(k) = t_{k-1}(k-1) + k-2 = {k-1\choose 2} + k-2\), we have

\begin{align*}
	S_k(K_k)	& \leq 2\cdot\SC(K_{k-1}) + 3\cdot 2^{k-3} \big(2^{{k-1\choose 2}} - \SC(K_{k-1})\big) \\
			 	& =  3\cdot 2^{k-3} 2^{{k-1\choose 2}} -\left( 3\cdot 2^{k-3} -2\right) \SC(K_{k-1}) \\
			 	& <  3\cdot 2^{k-4} 2^{{k-1\choose 2}} + 2^{{k-1\choose 2}}\\
			 	& \leq 2^{k-2} 2^{{k-1\choose 2}} = 2^{t_{r-1}(k)}.
\end{align*}

Analogously, from~\eqref{eq:skkk1}, since \(t_{k-1}(k+1) = t_{k-1}(k-1) + 2k-3 = {k-1\choose 2} + 2k-3\) we have

\begin{align*}
	S_k(K_{k+1})	& \leq 6\cdot\SC(K_{k-1}) + 2\cdot 3\cdot 2^{2k-5} \big(2^{{k-1\choose 2}} - \SC(K_{k-1})\big) \\
				 	& \leq 3\cdot 2^{2k-5} 2^{{k-1\choose 2}} +  3 \cdot2^{{k-1\choose 2}} \\
				 	& \leq 2^{2k-3} 2^{{k-1\choose 2}} = 2^{t_{r-1}(k+1)},
\end{align*}
which finishes the proof.
\end{proof}

\subsection{Proof of the main result}

For simplicity, we write $S_k(n)$ for $D(n,\S_k)$, i.e., $S_k(n) = \max\{{S}_k(G)\colon G\text{ is an $n$-vertex graph}\}$, which stands for the maximum number of orientations of $n$-vertex graphs with no strongly connected copies of $K_k$.
The following theorem is the main result of this paper.

\begin{theorem}
  \label{thm:Skn}
  Let $n\geq k\geq 4$
  Then,
  \[
    S_k(n) = \left\{
  \begin{array}{cl}
    40 & \text{ if } \,n=k=4,\nonumber\smallskip  \\
    2^{t_{k-1}(n)} & \text{ if } \,n\geq 5 \text{ or } \,k\geq 5.\nonumber
  \end{array}
\right.
\]
  Furthermore, $K_4$ is the only $\S_4$-extremal graph and, if $n\geq 5$ or $k\geq 5$, then the Tur\'an graph $T_{k-1}(n)$ is the only
  $\S_k$-extremal graph.
\end{theorem}

 Before proving Theorem~\ref{thm:Skn}, recall that Lemma~\ref{lemma:multi} implies that every $\F$-extremal graph is complete multipartite for the family $\F$ of $k$-vertex tournaments with no source.
Therefore, to prove Theorem~\ref{thm:Skn} (for $n\geq 5$ or $k\geq 5$), it is enough to show that every complete multipartite $\F$-extremal graph is the $(k-1)$-partite Tur\'an graph.

\begin{proof}[Proof of Theorem~\ref{thm:Skn}]
  The proof is by induction on \(n\).
  First, note that if $n<k$ then any graph \(G\) of order \(n\) is \(K_k\)-free,
and hence has \(2^{|E(G)|}\) \(S_k\)-free orientations. 
Therefore, \(S_k(n) = 2^{n \choose 2} = 2^{t_{k-1}(n)}\).

  For $n=k=4$ we used a simple computer program to verify that $S_4(4)=40$ and that $K_4$ is the only $\S_4$-extremal graph.
  Also with a (simple) computer program we verified that $S_4(n) =  2^{t_{3}(n)}$ for $5\leq n\leq 8$ and that $T_3(n)$ is the only $\S_4$-extremal graph.
  Thus, we assume that either $k=4$ and $n\geq 9$, or $n\geq k\geq 5$.

Since no tournament in $\S_k$ contains a source, Lemma~\ref{lemma:multi} implies that every $\S_k$-extremal graph is complete multipartite. Let $G$ be an $\S_k$-extremal $r$-partite graph and assume that $r\geq k$. First, suppose that \(G\) contains a clique \(K\) of size \(k+1\).
From Corollary~\ref{corollary:small-cliques}, we have for \(k\geq 4\) that \(S_k(K)<2^{t_{k-1}(k+1)}\).
Note that, since \(G\) is a complete multipartite graph,
if \(u\notin V(K)\),
then \(u\) is adjacent to either \(k\) or \(k+1\) vertices of \(K\).
Thus, by Lemma~\ref{lemma:cliqueExtension}, we have \(c_{\S_k}(u,\vec K) \leq 5\cdot 2^{k-3}\).
Therefore, we have

\begin{align}\label{eq:boundSk}
	S_k(G)	&\leq S_k(K) (5\cdot 2^{k-3})^{n-k-1} S_k(G\setminus K)\nonumber \\
			&< 2^{t_{k-1}(k+1) + (\log 5 + k-3)(n-k-1)}S_k(G\setminus K).
\end{align}

If \(k=4\) and \(n-k-1\neq 4\), or \(k\geq 5\)
then, by the induction hypothesis, we have \(S_k(G\setminus K)\leq 2^{t_{k-1}(n-k-1)}\).
Therefore, from \eqref{eq:boundSk} and Proposition~\ref{prop:turan-number-from-clique-k+1} we have

\begin{align*}
	S_k(G)	&< 2^{t_{k-1}(k+1) + (\log 5 + k-3)(n-k-1)}\cdot S_k(G\setminus K) <2^{t_{k-1}(n)},
\end{align*}

A contradiction with the fact that $G$ is $\S_k$-extremal.
Now, suppose \(k=4\) and \(n-k-1=4\).
In this case, \(S_k(G\setminus K)\leq 40\),
and hence, from~\eqref{eq:boundSk} we have $S_4(G) < 2^{t_3(5) + 4(\log 5 + 1)}\cdot S_k(G\setminus K)$, which implies

\begin{align*}
	S_4(G)	&<  2^{8 + 4(\log 5 + 1)+\log 40} < 2^{27} =2^{t_3(9)},
\end{align*}
a contradiction.
Therefore, no $\S_k$-extremal graph contains a clique with $k+1$ vertices.

Since $G$ contains no clique with $k+1$ vertices, we may and shall assume that \(G\) is \(k\)-partite.
We first deal with the case $k\geq 5$.
If \(n=k\), then \(G\simeq K_k\),
and hence, by Corollary~\ref{corollary:small-cliques}, 
we have \(S_k(G) < 2^{t_{k-1}(k)}\).
Thus, we may assume that \(n\geq k+1\).
Let \(K\) be a clique of size \(k\) in \(G\).
Since \(G\) is a complete \(k\)-partite graph,
if \(u\notin V(K)\),
then \(u\) is adjacent to precisely \(k-1\) vertices of \(K\).
Thus, by Lemma~\ref{lemma:cliqueExtension}, we have \(c_{\S_k}(u,\vec K) \leq 3\cdot 2^{k-3}\).
Therefore, we have 

\begin{align*}
	S_k(G)	&\leq S_k(K) (3\cdot 2^{k-3})^{n-k} \cdot S_k(G\setminus K). 
\end{align*}

Clearly, if \(G\setminus K\) is \((k-1)\)-partite, then \(S_k(G\setminus K)\leq 2^{t_{k-1}(n-k)}\);
and if \(G\setminus K\) is not \((k-1)\)-partite, 
then, by the induction hypothesis, we have \(S_k(G\setminus K)\leq 2^{t_{k-1}(n-k)}\).
Moreover, from Corollary~\ref{corollary:small-cliques}, we have \(S_k(K) < 2^{t_{k-1}(k)}\).
Therefore, by Proposition~\ref{prop:turan-number-from-clique-k} we have

\begin{align*}
	S_k(G)	&< 2^{t_{k-1}(k) + (\log 3 + k-3)(n-k)}\cdot S_k(G\setminus K) \\
			&\leq 2^{t_{k-1}(k) + (\log 3 + k-3)(n-k) + t_{k-1}(n-k)} \\
			&<2^{t_{k-1}(n)},
\end{align*}
which contradicts the fact that $G$ is $\S_k$-extremal.

It remains deal with the case for $k=4$ and $n\geq 9$ (where $G$ contains no clique with $k+1$ vertices).
For that, note that removing a copy of $K_4$ from $G$ results in a graph that is not a copy of $K_4$.
Then, by arguments analogous to the ones above (for $k\geq 5$), replacing Proposition~\ref{prop:turan-number-from-clique-k} with Proposition~\ref{prop:turan-number-from-clique-k=4} and using that $S_4(K_4) = 40$, we get $S_4(G) < 2^{t_{3}(n)}$, getting the desired contradiction.
\end{proof}

\section{Final remarks and open problems}
\label{sec:open}

In this paper, given a fixed family $\F$ of oriented graphs and a positive integer~$n$, we consider the quantity $D(n,\F)$, the maximum number of $\F$-free orientations of any $n$-vertex graph.
Given $k\geq 4$ and $\F \in \{\S_k, \R_k\}$, where $\S_k$ and $\R_k$ are, respectively, the families of $k$-vertex tournaments that are strongly connected and non-transitive, we determined $D(n,\F)$ and the corresponding extremal $n$-vertex graphs for all values of~ $n$. For $k=3$, we have $\S_3 =  \R_3 = \{\circcycle_3\}$ and Theorem~\ref{conj:AY} determines the value of $D(n,\{\circcycle_3\})$ for all values of $n$.  

This type of extremal problem was first investigated by Alon and Yuster~\cite{AlYu06}, who showed a that the Tur\'{a}n graph $T_{k-1}(n)$ is the unique $\{\vec{K}_k\}$-extremal graph for sufficiently large $n$.
For more than 10 years, there had been no substantial contribution to this problem, but this has changed in the past year~\cite{ArBoMo20+,BuSu20+}.
This renewed interest in this problem comes at a time of many advances in a related problem about the number of distinct edge-colorings avoiding some fixed color-patterns, known as the Erd\H{o}s-Rothschild problem (see, e.g.,~\cite{BL19, BBH20, CGM20, rainbow_kn, PSY16}).

In the remainder of this section, we raise a few questions whose answers might substantially improve our understanding of $D(n,\F)$ and of the corresponding $n$-vertex $\F$-extremal graphs.

Given an integer $k \geq 3$ and a $k$-vertex tournament $\vec{K}_k$, define $n_0(\vec{K}_k)$ as follows:
if $T_{k-1}(n)$ is the unique $\vec{K}_k$-extremal graph for every $n \geq 1$, we set $n_0(\vec{K}_k)=1$.
Otherwise, set $n_0(\vec{K}_k) = n_0$, where $n_0$ is an integer such that $T_{k-1}(n)$ is the unique $n$-vertex $\vec{K}_k$-extremal graph for every $n \geq n_0$, but $T_{k-1}(n_0 - 1)$ is not the unique $\vec{K}_k$-extremal graph with $n_0-1$ vertices.
For the next problem, let  $\mathcal{O}_k$ be the set of all orientations of $K_k$ and define $n_{\max}(k) = \max\{n_0(\vec{K}_k) \colon \vec{K}_k \in \mathcal{O}_k \}$.

\begin{problem}\label{prob1}
Given an integer $k \geq 3$, determine $n_{\max}(k)$ and characterize the orientations $\vec{K}_k$ for which $n_0(\vec{K}_k) = n_{\max}(k)$. 
\end{problem}  

For $k=3$, the answer to Problem~\ref{prob1} is known. The only orientation of $K_3$ other than $\circcycle_3$ is the transitive tournament $\vec{T}_3$. It is easy to see that $n_0(\vec{T}_3)=1$ and Theorem~\ref{conj:AY} shows that $n_0(\circcycle_3)=8$, so that $n_{\max}(3)=8$.
However, the answer to Problem~\ref{prob1} for $k\geq 4$ remains open. 

\begin{problem}\label{prob2}
  Determine families of oriented graphs $\F$ for which all $n$-vertex $\F$-extremal graphs are complete multipartite and verify whether these graphs are necessarily balanced.
\end{problem}   

In connection with Problem~\ref{prob2}, Lemma~\ref{thm_existence} states that if $\F$ is an arbitrary family of \emph{tournaments}, then, for every $n$, \emph{there exists} an $n$-vertex $\F$-extremal graph that is complete multipartite. For specific families $\F$ of tournaments, Lemma~\ref{lemma:multi} does more and establishes that all $\F$-extremal graphs are complete multipartite. However, neither proofs ensure balancedness, even though all $\F$-extremal graphs known so far are balanced complete multipartite graphs\footnote{At least in situations where the (Tur\'{a}n) extremal graphs of the graphs whose orientations are in $\F$ are complete multipartite graphs.}.

For the next problem given a graph $F$, we say an $n$-vertex graph $H$ is $F$-extremal if $|E(H)|$ is maximum among all $n$-vertex graphs that do not contain a copy of $F$.

\begin{problem}\label{prob3}
Determine the graphs $F$ such that there exists an orientation $\vec{F}$ with the property that, for arbitrarily large values of $n$, there are $n$-vertex $\vec{F}$-extremal graphs that are not $F$-extremal.
\end{problem}   

As mentioned before, it is known that no complete graph satisfies the property described in Problem~\ref{prob3}. However, if $F$ is the complete bipartite graph $K_{1,3}$ and $\vec{F}$ is the orientation where all edges are oriented towards its leafs, it is easy to see that, for any positive integer $n$ divisible by 4, a graph given by $n/4$ disjoint copies of $K_4$ admits $32^{n/4}=2^{5n/4}$ distinct $\vec{F}$-free copies, while any $K_{1,3}$-extremal graph (which has maximum degree at most $2$) admits at most $2^{n}$ orientations.

We present now one final open problem concerning the uniqueness of extremal graphs.

\begin{problem}\label{prob4}
  Determine the families of orientations $\F$ with the property that, for all $n\geq 1$, there is a unique $n$-vertex $\vec F$-extremal graph.
\end{problem}

We have seen that the families $\S_k$ and $\R_k$ satisfy the property described in Problem~\ref{prob4}.
As a consequence of the proof of~\cite[Theorem 1.1]{AlYu06}, at least for large $n$, any family $\F$ of tournaments admits a unique $n$-vertex $\F$-extremal graph.
However, this does not necessarily extend to all values of $n$. For instance, consider the family $\mathcal{U}_4$ of all orientations of $K_4$ that have no source. Then $K_4$ and $T_3(4)$ both admit 32 distinct $\mathcal{U}_4$-free orientations, and are $\mathcal{U}_4$-extremal graphs on four vertices. Since $\S_4 \subset \mathcal{U}_4$, Theorem~\ref{thm:Skn} implies that $T_3(n)$ is the unique $n$-vertex $\mathcal{U}_4$-extremal graph for all $n \geq 5$.

\bibliographystyle{amsplain}
\providecommand{\bysame}{\leavevmode\hbox to3em{\hrulefill}\thinspace}
\providecommand{\MR}{\relax\ifhmode\unskip\space\fi MR }
\providecommand{\MRhref}[2]{%
  \href{http://www.ams.org/mathscinet-getitem?mr=#1}{#2}
}
\providecommand{\href}[2]{#2}

\section*{Appendix}
\begin{proposition}\label{prop:turan-number-from-clique-k+1}
	Let \(n\) and \(k\) be positive integers such that \(n> k\geq 4\).
	Then \[t_{k-1}(k+1) + (\log 5 + k -3)(n-k-1) + t_{k-1}(n-k-1) < t_{k-1}(n).\]
\end{proposition}

\begin{proof}
	First, let \(q\) and \(r\) be such that \(n = (k-1)q + r\), where \(0\leq r< k-1\),
and let \(q'\) and \(r'\) be such that \(n-k-1 = (k-1)q' + r'\), where \(0\leq r'< k-1\).
Note that \(r'\equiv r-2\pmod{k-1}\),
and hence \((q',r')\in\{(q-2,k-3),(q-2,k-2),(q-1,r-2)\}\),
where \(r' = r-2\) if and only if \(r\geq 2\).
Also \(t_{k-1}(k+1) = {k+1 \choose 2} - 2\) and
\(t_{k-1}(n-k-1) = {n-k-1 \choose 2} - r'{q'+1\choose 2} - (k-1-r'){q'\choose 2}\).
Thus, we have

\begin{align*}
	& t_{k-1}(k+1) + (\log 5 + k-3)(n-k-1) + t_{k-1}(n-k-1) \\
	& = {k+1\choose 2} -2 + (k+1)(n-k-1) + (\log 5 -4)(n-k-1) \\
	& + {n-k-1\choose 2} - r'{q'+1\choose 2} - (k-1-r'){q'\choose 2} \\
	& = {n\choose 2}-2+ (\log 5 -4)(n-k-1)- r'{q'+1\choose 2} - (k-1-r'){q'\choose 2} \\
	& ={n\choose 2}-2+ (\log 5 -3)(n-k-1)- r'{q'+2\choose 2} - (k-1-r'){q'+1\choose 2},
\end{align*}
where in the last equality we used that \(n-k-1 = r'(q'+1) + (k-1-r')q'\).

Since \(k\geq 4\), we have

\begin{align} \label{eq_aux}
(\log 5 -3)(n-k-1) \nonumber
&= (\log 5 -3)((k-1)q'+r') \\
&\stackrel{(\ast)}{\leq} (\log 5 -3)(k-1)q' \\ \nonumber
&\leq (\log 5 -3)3q' \\ \nonumber
&<  -2q', 
\end{align}
and hence, 
\begin{align*}
	& {n\choose 2}-2+ (\log 5 -3)(n-k-1)- r'{q'+2\choose 2} - (k-1-r'){q'+1\choose 2} \\
	& <	{n\choose 2}-2 - 2q'- r'{q'+2\choose 2} - (k-1-r'){q'+1\choose 2} 
\end{align*}

In what follows, we divide the proof according to \((q',r')\).
If \((q',r') = (q-1,r-2)\), then \(r\geq 2\) and we have

\begin{align*}
	&{n\choose 2}-2 - 2q'- r'{q'+2\choose 2} - (k-1-r'){q'+1\choose 2} \\
	& = {n\choose 2}  - (r'+2){q'+2\choose 2} - (k-1-r'-2){q'+1\choose 2} \\
	& = {n\choose 2}- r{q+1\choose 2} - (k-1-r){q\choose 2} \\
	& = t_{k-1}(n).
\end{align*}

If \((q',r') = (q-2,k-3)\), then \(r=0\) and we have

\begin{align*}
	& 	{n\choose 2}-2 - 2q'- r'{q'+2\choose 2} - (k-1-r'){q'+1\choose 2} \\
	&= 	{n\choose 2}-2 - 2q'- (k-3){q\choose 2} - 2{q-1\choose 2} \\
	& = {n\choose 2} - (k-1){q\choose 2} \\
	& = t_{k-1}(n).
\end{align*}

If \((q',r') = (q-2,k-2)\), then \(r=1\) and we obtain the following variation of~\eqref{eq_aux}. 

\begin{align*}
(\log 5 -3)(n-k-1) &= (\log 5 -3)((k-1)q'+r') \\ 
&\leq (\log 5 -3)\left((k-1)q'+(k-2)\right) \\
&\leq (\log 5 -3) (3q'+2) <-2q'-1
\end{align*}

Therefore,
\begin{align*}
	& {n\choose 2}-2+ (\log 5 -3)(n-k-1)- r'{q'+2\choose 2} - (k-1-r'){q'+1\choose 2} \\
	&< 	{n\choose 2}-2 -2q'-1- r'{q'+2\choose 2} - (k-1-r'){q'+1\choose 2} \\
	&= 	{n\choose 2}-3 - 2q'- (k-2){q\choose 2} - {q-1\choose 2} \\
	& = {n\choose 2}+1-2q - (k-2){q\choose 2} - {q-1\choose 2} \\
	& = {n\choose 2} - (k-1-r){q\choose 2} - r{q+1\choose 2}\\
	& = t_{k-1}(n).
\end{align*}
\end{proof}

\begin{proposition}\label{prop:turan-number-from-clique-k}
	Let \(n\) and \(k\) be positive integers such that \(n> k\geq 4\).
	Then \[t_{k-1}(k) + (\log 3 +k-3)(n-k) + t_{k-1}(n-k) < t_{k-1}(n).\]
\end{proposition}

\begin{proof}
	First, let \(q\) and \(r\) be such that \(n = (k-1)q + r\), where \(0\leq r< k-1\),
and let \(q'\) and \(r'\) be such that \(n-k = (k-1)q' + r'\), where \(0\leq r'< k-1\).
Note that \(r'\equiv r-1\pmod{k-1}\),
and hence \((q',r')\in\{(q-2,k-2),(q-1,r-1)\}\),
where \(r' = k-2\) if and only if \(r=0\).
Note that \(t_{k-1}(k) = {k \choose 2} - 1\) and
\(t_{k-1}(n-k) = {n-k \choose 2} - r'{q'+1\choose 2} - (k-1-r'){q'\choose 2}\).
Thus, we have

\begin{align*}
	& t_{k-1}(k) + (\log 3 +k-3)(n-k) + t_{k-1}(n-k) \\
	& = {k\choose 2} -1 + k(n-k) + (\log 3 -3)(n-k) 
	 + {n-k\choose 2} - r'{q'+1\choose 2} - (k-1-r'){q'\choose 2} \\
	& = {n\choose 2}-1+ (\log 3 -3)(n-k)- r'{q'+1\choose 2} - (k-1-r'){q'\choose 2}
\end{align*}

Since \(n-k = r'(q'+1) + (k-1-r')q'\),
we have

\begin{align*}
	& {n\choose 2}-1 +(\log 3 -3)(n-k)- r'{q'+1\choose 2} - (k-1-r'){q'\choose 2} \\
	& ={n\choose 2}-1+(\log 3 -2)(n-k)- r'{q'+2\choose 2} - (k-1-r'){q'+1\choose 2}
\end{align*}

Since \(k\geq 4\), we have 
\[
(\log 3 -2)(n-k) = (\log 3 -2)((k-1)q'+r') \leq (\log 3 -2)(k-1)q' \leq (\log 3 -2)3q' <  -q', 
\]
In what follows, we divide the proof according to \((q',r')\).
If \((q',r') = (q-1,r-1)\), then \(r\geq 1\) and we have

\begin{align*}
	& {n\choose 2}-1+(\log 3 -2)(n-k)- r'{q'+2\choose 2} - (k-1-r'){q'+1\choose 2} \\
	& <{n\choose 2}-1-q'- r'{q'+2\choose 2} - (k-1-r'){q'+1\choose 2} \\
	& ={n\choose 2}- (r'+1){q'+2\choose 2} - (k-1-r'-1){q'+1\choose 2} \\
	& = {n\choose 2}- r{q+1\choose 2} - (k-1-r){q\choose 2} \\
	& = t_{k-1}(n).
\end{align*}

If \((q',r') = (q-2,k-2)\), then \(r=0\) and we have

\begin{align*}
	& 	{n\choose 2}-1 - q'- r'{q'+2\choose 2} - (k-1-r'){q'+1\choose 2} \\
	& =	{n\choose 2}-1-q'- (k-2){q'+2\choose 2} - {q'+1\choose 2}\\
	& =	{n\choose 2}- (k-1){q\choose 2}\\
	& = t_{k-1}(n).
\end{align*}
\end{proof}

\begin{proposition}\label{prop:turan-number-from-clique-k=4}
	Let \(n\) be a positive integer such that \(n\geq 9\).
	Then \[\log 40 + (\log 3 +1)(n-4) + t_{3}(n-4) < t_{3}(n).\]
\end{proposition}

\begin{proof}
	First, let \(q\) and \(r\) be such that \(n = 3q + r\), where \(0\leq r< 3\),
and let \(q'\) and \(r'\) be such that \(n-4 = 3q' + r'\), where \(0\leq r'< 3\).
Note that \(r'\equiv r-1\pmod{3}\),
and hence \((q',r')\in\{(q-2,2),(q-1,r-1)\}\),
where \(r' = 2\) if and only if \(r=0\).
Note that \(t_{3}(n-4) = {n-4 \choose 2} - r'{q'+1\choose 2} - (3-r'){q'\choose 2}\).
Thus, we have

\begin{align*}
	& \log 40 + (\log 3 +1)(n-4) + t_{3}(n-4) \\
	& = 6 + (\log 40-6) + (\log 3 -3)(n-4) + 4(n-4)
	 + {n-4\choose 2} - r'{q'+1\choose 2} - (3-r'){q'\choose 2} \\
	& = {n\choose 2} + (\log 40-6)+ (\log 3 -3)(n-4)- r'{q'+1\choose 2} - (3-r'){q'\choose 2}
\end{align*}

Since \(n-4 = r'(q'+1) + (3-r')q'\),
we have

\begin{align*}
	& {n\choose 2}+ (\log 40-6) +(\log 3 -3)(n-4)- r'{q'+1\choose 2} - (3-r'){q'\choose 2} \\
	& ={n\choose 2}+ (\log 40-6)+(\log 3 -2)(n-4)- r'{q'+2\choose 2} - (3-r'){q'+1\choose 2}
\end{align*}

It is not hard to check that, since \(n\geq 9\), we have \(r'>0\) or \(q'\geq 2\),
and hence

\[
(\log 40-6)+(\log 3 -2)(n-4) = (\log 40-6)+(\log 3 -2)(3q'+r') < -q'-1, 
\]

In what follows, we divide the proof according to \((q',r')\).
If \((q',r') = (q-1,r-1)\), then \(r\geq 1\) and we have

\begin{align*}
	& {n\choose 2}+ (\log 40-6)+(\log 3 -2)(n-4)- r'{q'+2\choose 2} - (3-r'){q'+1\choose 2} \\
	& <{n\choose 2}-1-q'- r'{q'+2\choose 2} - (3-r'){q'+1\choose 2} \\
	& ={n\choose 2}- (r'+1){q'+2\choose 2} - (2-r'){q'+1\choose 2} \\
	& = {n\choose 2}- r{q+1\choose 2} - (3-r){q\choose 2} \\
	& = t_{k-1}(n).
\end{align*}

If \((q',r') = (q-2,2)\), then \(r=0\) and we have

\begin{align*}
	& 	{n\choose 2}- (r'+1){q'+2\choose 2} - (2-r'){q'+1\choose 2} \\
	& =	{n\choose 2}- 3{q'+2\choose 2}  \\
	& =	{n\choose 2}- 3{q\choose 2}\\
	& = t_{k-1}(n).
\end{align*}
\end{proof}

\end{document}